\documentclass{dcds}
\usepackage{amsmath}
  \usepackage{paralist}
  \usepackage{graphics} 
  \usepackage{epsfig} 
\usepackage{graphicx}  \usepackage{epstopdf}
 \usepackage[colorlinks=true]{hyperref}
\usepackage{mathrsfs,amsfonts,amssymb,amsmath}
\usepackage{cite}
\hypersetup{urlcolor=blue, citecolor=red}

\def\currentvolume{39}
 \def\currentissue{11}
  \def\currentyear{2019}
   \def\currentmonth{November}
    \def\ppages{X--XX}
     \def\DOI{10.3934/dcds.2019286}

\newtheorem{theorem}{Theorem}[section]
\newtheorem{corollary}[theorem]{Corollary}

\newtheorem{lemma}[theorem]{Lemma}
{\theoremstyle{definition}\newtheorem{definition}[theorem]{Definition}
\newtheorem{remark}[theorem]{Remark}
\newtheorem{notation}[theorem]{Notation}}

\newcommand{\Lip}{\mbox{\rm Lip}}

\title[Invariance Principle for Intermittent Non-stationary System] 
      {Almost surely invariance principle for non-stationary and random intermittent dynamical systems}

\author[Yaofeng Su]{}

 \email{yfsu@math.uh.edu}

\begin{document}
\maketitle


\centerline{\scshape Yaofeng Su}
\medskip
{\footnotesize
 \centerline{Department of Mathematics}
 \centerline{ University of Houston}
   \centerline{ Houston, Texas 77204-3008, USA}
} 


\bigskip

 \centerline{(Communicated by Zhiren Wang)}

\begin{abstract}
We establish almost sure invariance principles (ASIP), a strong form of approximation by Brownian motion, for non-stationary time series arising as observations on sequential maps possessing an indifferent fixed point. These transformations are obtained by perturbing the slope in the Pomeau-Manneville map. Quenched ASIP for random compositions of these maps is also obtained.
\end{abstract}

\section{Introduction}

Almost surely invariance principle (ASIP) is a very strong statistical property. It is a matching of the trajectories of the dynamical system with a Brownian motion in such a way that the error is negligible in comparison with the Birkhoff sum. Limit theorems such as the central limit theorem (CLT), the functional central limit theorem and the law of the iterated logarithm (LIL) transfer from the Brownian motion to time-series generated by observations on the dynamical system.

Haydn, Nicol, T\"or\"ok, Vaienti \cite {HNTV} dealt with ASIP for a non-stationary process given by the observation along the orbit obtained by concatenating maps chosen in a given set. They are in one and more dimensions a.e. piecewise expanding, more precisely their transfer operator (Perron-Frobenius operator) with respect to the Lebesgue measure is quasi-compact on a suitable Banach space. This allows
to approximate the original process with reverse martingale differences plus an error. Same approach is applied to random dynamical system with sufficient hyperbolicity in \cite {DFGV}. Based on Skorokhod embedding technique, Cuny and Merlev\`ede  \cite {CM} recently showed that reverse martingale differences satisfy ASIP under some conditions. The error is shown to be essentially bounded due to the presence of a spectral gap in the transfer operator on a Banach space continuously injected in $L^{\infty}$. Therefore ASIP for non-stationary dynamical systems in \cite {HNTV} and quenched ASIP for random dynamical systems in  \cite{DFGV} are satisfied.

However, if the transfer operator with respect to the Lebesgue measure is not quasi-compact on a suitable Banach space, the approach described above fails to work. Such example for non-stationary dynamical systems and related statistical properties are provided in a preceding series of papers  \cite{AHNTV, NTV, OJ}.

The first paper \cite{AHNTV} considered composition of Pomeau-Manneville like maps, obtained
by perturbing the slope at the indifferent fixed point 0. They obtain polynomial decay of correlations for particular classes of centered observables, which could also be interpreted as the decay of the iterates of the transfer operator on functions of zero (Lebesgue) average; this fact is also known as loss of memory.

The last two papers \cite{NTV, OJ} considered the same system and proved self-norming CLT under the assumption that it is sufficiently chaotic and the variance grows at a certain rate. Moreover, they proved self-norming CLT for nearby maps and quenched CLT for random compositions of maps in the same family provided the system is sufficiently chaotic and the base map of random dynamical system has strong mixing.

In this paper, the same system as \cite{NTV} is considered and some of its properties are improved, namely the stronger statistical property (ASIP) is obtained. Our construction for Gaussian variable in ASIP is close to Proposition 2.1 in \cite{CM}, that is, applying Skorohod embedding to tail series, but we won't impose strong conditions like (2.1), (2.2) in \cite{CM}, which loses lots of information of Skorohod embedding. Instead, we will give a sharp condition for ASIP (see our Lemma \ref{criteria1}). Surprisingly, this condition can be verifed by our system considered here. Besides,  due to non-uniformity of our system, the error rate of our ASIP is just slightly less than $\frac{1}{2}$ (not $\frac{1}{4}$ in \cite{CM}). So we will not give an explicit formula for it.

\section{ Definitions and notations}\label{positive}

Consider a family of Pomeau-Manneville maps on $[0,1]$: $  0< \alpha <1 $,
\begin{eqnarray}
T_{\alpha}(x) =
\begin{cases}
x+2^{\alpha}x^{1+\alpha},      & 0\le x \le \frac{1}{2}\\
2x-1,  & \frac{1}{2} < x \le 1 \\
\end{cases}.
\end{eqnarray}

Given $ n, m, k \in \mathbb{N}, 0< \beta_k < \alpha $, denote:
\[T_k := T_{\beta_{k}},\]
\[T^{n+m}_m :=T_{n+m} \circ T_{n+m-1} \circ \dots\circ T_m,\]
\[T^n :=T_1^{n}=T_n \circ T_{n-1} \circ \dots\circ T_1.\]

The transfer operator (Perron-Frobenius operator) $P_k$ associated to $T_k$ is defined by the
duality relation:
{\small\[\int g \cdot P_{k} f dm= \int g \circ T_{k} \cdot f dm \text{ for all } f
  \in L^1, g \in L^{\infty}, \text{ where } dm \text{ is Lebesgue
    measure.}\]}

Similar to $ T_k $, denote:
\[P_k := P_{\beta_{k}},\]
\[P^{n+m}_m :=P_{n+m} \circ P_{n+m-1} \circ \dots\circ P_m,\]
\[P^n :=P_1^{n}=P_n \circ P_{n-1} \circ \dots\circ P_1.\]

As in \cite {AHNTV}, define $X(x) :=  x, x\in [0,1]$, a cone $ C_a \subset{L^1[0,1]}$ by:
\[ C_a :=\{f \in C^1(0,1]: f \ge 0, f \text{ decreasing, } X^{\alpha+1} \cdot f \text{ increasing, } f(x) \le ax^{-\alpha}\int { f} dm\} \]

From \cite {AHNTV}, if $a$ is chosen large enough, $ C_a$ is preserved by every $ P_n $, hence by every
$ P_n^{n+m}$. We fix such a large $a$ from now on. The following decay of correlations holds:

\begin{theorem}[see \cite{AHNTV, NTV}]\label{citethm} \par
  Assume $K, M >0$, $\phi_k \in C^1[0,1]$ and $h_k\in C_a$ with
  $ ||\phi_k||_{C^1}\le K, ||h_k||_{L^1} \le M$ for all $k \ge 1 $.

  Then, for $1\le p< \frac{1}{\alpha}$, there is a constant
  $ C_{K,M, \alpha,p} $ such that for all $m, n \in \mathbb{N}$:
\[ ||P_{m+1}^{n+m}(\phi_k\cdot h_k -\int{\phi_k\cdot h_k}dm)||_{L^p}\le C_{K,M, \alpha,p} \cdot \frac{1}{n^{\frac{1}{p\alpha}-1}} \cdot (\log n)^{\frac{1-p\alpha}{\alpha p-p \alpha^2}}.\]
\end{theorem}

\begin{corollary} \label{decay1}
  Assume $ K, M >0$, $ \phi_k \in \Lip[0,1]=W^{1,\infty}([0,1])$ and
  $h_k\in C_a$ are s.t.
  $ ||h_k||_{L^1} \le M,||\phi_k||_{W^{1,\infty}} :=
  ||\phi_k||_{L^{\infty}}+\sup_{x \not= y\in[0,1]}\frac{|\phi_k(x)-\phi_k(y)|}{|x-y|}\le K$ for all $k \ge 1 $.

  Then, for $ 1\le p< \frac{1}{\alpha}$, there is a constant
  $ C_{K,M, \alpha,p} $ such that for all $m, n \in \mathbb{N}$:
\begin{equation} \label{decay}
||P_{m+1}^{n+m}(\phi_k\cdot h_k -\int{\phi_k\cdot h_k}dm)||_{L^p}\le C_{K,M, \alpha,p} \cdot \frac{1}{n^{\frac{1}{p\alpha}-1}}\cdot (\log n)^{\frac{1-p\alpha}{\alpha p-p \alpha^2}}.
\end{equation}
\end{corollary}

For simplicity, in many of the following statements
$\frac{1}{n^{\frac{1}{p\alpha}-1}}$ will be used as the rate of decay, ignoring the $\log n$-factor. This is still correct if taking $\alpha$ a slightly larger value (and is actually the correct rate of decay for the stationary case).

\begin{proof}
Using convolution, for $ \phi_k \in \Lip[0,1]$, there is $\phi_{k,\epsilon} \in C^{\infty}[0,1]$ such that
\[||\phi_{k,\epsilon}-\phi_k||_{L^{\infty}}\le \epsilon||\phi_k||_{W^{1,\infty}}\le K \epsilon,\]
\[||\phi_{k,\epsilon}||_{W^{1,\infty}}\le||\phi_k||_{W^{1,\infty}}\le K.\]

By Theorem \ref{citethm},
\[||P_{m+1}^{n+m}(\phi_{k,\epsilon}\cdot h_k -\int{\phi_{k,\epsilon}\cdot h_k}dm)||_{L^p}\le C_{K,M,\alpha, p}\cdot\frac{1}{n^{\frac{1}{p\alpha}-1}}\cdot (\log n)^{\frac{1-p\alpha}{\alpha p-p \alpha^2}}.\]

Let $\epsilon\to 0$,
\[||P_{m+1}^{n+m}(\phi_{k,\epsilon}\cdot h_k -\int{\phi_{k,\epsilon}\cdot h_k}dm)-P_{m+1}^{n+m}(\phi_k\cdot h_k -\int{\phi_k\cdot h_k}dm)||_{L^p}\]
\[\le \epsilon \cdot ||\phi_k||_{W^{1,\infty}} \cdot (||P_{m+1}^{n+m}{h_k}||_{L^p}+||h_k||_{L^1}\cdot||P_{m+1}^{m+n} \mathbf{1}||_{L^p})\]
\[\le a \cdot \epsilon \cdot K \cdot(1+||h_k||_{L^1})\cdot (\int_{[0,1]}{\frac{1}{x^{\alpha p}}dm})^{\frac{1}{p}}\to 0.\]
The last inequality holds since $P_{m+1}^{n+m}h_k, P_{m+1}^{n+m}\mathbf{1} \in C_a$. Hence (\ref{decay}) holds.
\end{proof}

\begin{definition}[ASIP for a Non-stationary Dynamical System]
  Given a non-stationary dynamical system
  $ ([0,1], \mathcal{B}, (T_k)_{k \ge 1}, dm)$ and an observation
  $ \phi \in \Lip[0,1]$, denote $\phi_k:=\phi-\int{\phi \circ T^k dm}$.

  Then $ (\phi_k \circ T^k)_{k \ge 1} $ satisfies ASIP if there are
  $\epsilon \in (0,1) $ and independent mean zero Gaussian
  variables $ (G_k)_{k \ge 1}$ on some extended probability space of
  $([0,1], \mathcal{B},dm)$ such that almost surely,
\[ \sum_{k\le n} \phi_k \circ T^k - \sum_{k \le n} G_k=o(\Sigma_n^{1-\epsilon}),\]
\[\sum_{k\le n} \mathbb{E}{ G_k^2} =\Sigma_n^2+ O(\Sigma_n^{2(1-\epsilon)}),\]
\[ \text{with }\Sigma_n^2:=\int(\sum_{k\le n} \phi_k \circ T^k )^2 dm \to \infty.\]

\end{definition}

We point out that if $ (\phi_k \circ T^k)_{k \ge 1} $ satisfies ASIP, then
it also satisfies the self-norming CLT and LIL:
\[\frac{\sum_{k\le n} \phi_k \circ T^k}{\Sigma_n} \stackrel{d}{\to} N(0,1),\]
\[\limsup_{n\to \infty}\frac{\sum_{k\le n} \phi_k \circ T^k}{\sqrt{\Sigma_n^2\log\log\Sigma_n^2}}=1,\]
\[\liminf_{n\to \infty}\frac{\sum_{k\le n} \phi_k \circ T^k}{\sqrt{\Sigma_n^2\log\log\Sigma_n^2}}=-1.\]
In fact, there is a matching of the Birkhoff sums
$\sum_{k\le n} \phi_k \circ T^k$ with a standard Brownian motion $B_t $
observed at times of order $\Sigma_n^2$ so that $\sum_{k\le n} \phi_k \circ T^k$
equals $B_{\Sigma_n^2}$ plus a negligible error almost surely.

\begin{notation} $a_n \approx b_n$(resp. $a_n \precsim b_n$) means there is a constant $C \ge 1$ such that $ C^{-1} \cdot b_n \le a_n \le C \cdot b_n$ for all $n$ (resp. $a_n \le C \cdot b_n $ for all $n$).
\end{notation}

\section{The main theorem}

Our main theorem is the following:

\begin{theorem}[ASIP]\ \label{thm}
  Consider the non-stationary dynamical system
  $ ([0,1],\break \mathcal{B}, (T_k)_{k \ge 1}, dm)$, where $dm $ is Lebesgue
  measure, $T_k: =T_{\beta_k}$ are Pomeau-Manneville maps,
  $ 0<\beta_k< \alpha<1 $, and an observation $\phi \in \Lip[0,1]$. If
  $ \Sigma_n^2 := \int (\sum_{k \leq n } \phi_k \circ T^k)^2dm \succsim
  n^{\gamma}$, then
\[ ( \phi_k \circ T^k )_{k \ge 1}\text{ satisfies ASIP} \text{ when } \gamma > \frac{1}{2} + \frac{1+2\alpha}{4(1-2\alpha)}, \alpha < \frac{1}{8}.\]
\end{theorem}

The main steps of the proof are the following:

\vspace*{4pt}\noindent\textbf{First step.} Decompose $\sum_{k \leq n } \phi_k \circ T^k$ as reverse
martingale differences plus error term.

\vspace*{4pt}\noindent\textbf{Second step.} Prove the error term has uniform $ L^p$-bound.

\vspace*{4pt}\noindent\textbf{Third step.} Apply Skorokhod embedding to reverse martingale differences,
obtain a sequence of suitable Gaussian variables from Brownian motion.

\vspace*{4pt}\noindent\textbf{Fourth step.} Prove that when $ \gamma > \frac{1}{2} +
\frac{1+2\alpha}{4(1-2\alpha)}, \alpha < \frac{1}{8}$, then the ASIP is satisfied.

First, we will cite/improve some lemmas below:

\subsection*{Step 1: Decomposition}

\begin{lemma}[\cite{NTV}]  For all $n \ge 1$, define
  $ H_{n+1} \circ T^{n+1} := \mathbb{E}[\sum_{k\le n} \phi_k \circ
  T^k|T^{-(n+1)} \mathcal{B}]$. Then there are reverse martingale differences
  $(\psi_k \circ T^k)_{k \ge 1}$ w.r.t. decreasing filtration
  $(T^{-k} \mathcal{B})_{k \ge 1}$ such that:
\begin{equation}\label{martingale}
\sum_{k\le n} \phi_k \circ T^k=\sum_{k \leq n} \psi_k \circ T^k+H_{n+1} \circ T^{n+1},
\end{equation}
\begin{equation}\label{variance}
\sigma_n^2 = \int (\sum_{k \leq n} \psi_k \circ T^k )^2 dm + \int H_{n+1}^2\circ T^{n+1}dm,
\end{equation}
\begin{equation}\label{error}
H_{n+1} \circ T^{n+1}=\frac{\sum_{k\le n}P_{k+1}^{n+1}(\phi_k \cdot P^k\mathbf{1})}{P^{n+1}\mathbf{1}}\circ T^{n+1}.
\end{equation}
\end{lemma}

\subsection*{Step 2: Uniform bound}

\begin{lemma}[lower bound, see \cite{AHNTV}]\label{lower bound}
\[\inf_{n,x} P^n\mathbf{1}(x) > 0.\]
\end{lemma}
\begin{lemma}[upper bound]
\begin{equation}\label{bound1}
\sup_{n}||H_n||_{L^r} < \infty, \text{ if } 1\le r < \frac{1}{2\alpha},
\end{equation}
\begin{equation}\label{bound2}
\sup_{n}||H_n \circ T^n||_{L^r} < \infty, \text{ if } 1\le r < \frac{1}{2\alpha},
\end{equation}
\begin{equation}\label{bound3}
\sup_n||\psi_n \circ T^n||_{L^r} < \infty, \text{ if }  1\le r < \frac{1}{2\alpha}.
\end{equation}
\end{lemma}

\begin{proof} (See also the Note Added in Proof in \cite{NTV})

By (\ref{decay}), (\ref{error}), Lemma \ref{lower bound}, and $ P^{n+1}\mathbf{1} \in C_a$, let $ 1 \le r< \frac{1}{2 \alpha}$:
\[\int{ |H_{n+1} \circ T^{n+1}|^r}dm=\int{|H_{n+1}|^r \cdot P^{n+1}\mathbf{1}}dm\]
\[=\int{|\sum_{i\le n}P_{i+1}^{n+1}(\phi_i\cdot P^i\mathbf{1})|^r \cdot |P^{n+1}\mathbf{1}|^{1-r}}dm\]
\[\le (\int{|\sum_{i\le n}P_{i+1}^{n+1}(\phi_i\cdot P^i\mathbf{1})|^{r}}dm)\cdot ||(P^{n+1}\textbf{1})^{1-r}||_{L^{\infty}}\precsim (\sum_{i \le n} \frac{1}{i^{\frac{1}{r\alpha}-1}})^r < \infty .\]

So (\ref{bound2}) \text{holds when} $1 \le r< \frac{1}{2\alpha}.$

To prove (\ref{bound1}), let $1\le r< \frac{1}{2\alpha}$, by Lemma \ref{lower bound}, we have:
\[\sup_n \int{|H_{n+1}|^r} dm \precsim \sup_n\int{|H_{n+1}|^r  \cdot P^{n+1} \mathbf{1}dm}=\sup_n\int{ |H_{n+1} \circ T^{n+1}|^r} dm < \infty.\]

By (\ref{martingale}), (\ref{bound2}) and $ \sup_n||\phi_n||_{L^\infty}<\infty$: $\sup_n||\psi_n \circ T^n||_{L^r} < \infty, \text{ when }  1\le r < \frac{1}{2\alpha}.$
\end{proof}

\begin{lemma}\label{close variance}
\[\sum_{k\le n}\int \psi_k^2 \circ T^k dm =\Sigma_n^2 +O(1), \text{ when } \alpha < \frac{1}{4}.\]
\end{lemma}

\begin{proof}
Take $r=2$ in (\ref{bound2}), i.e. $ \alpha < \frac{1}{4}$, (\ref{variance}) becomes:
\[\Sigma_n^2 = \int (\sum_{k \leq n} \psi_k \circ T^k )^2 dm + ||H_{n+1}\circ T^{n+1}||_{L^2}^2= \sum_{k \leq n} \int \psi_k^2 \circ T^k dm +O(1).\]
\end{proof}

\section{Proof of Theorem \ref{thm} }

\subsection*{Step 3: ASIP criteria}

The Skorokhod embedding will be used to match
$ \sum_{k \le n} \psi_k\break \circ T^k$ with a Brownian motion. For convenience, we define the following notation from now on:
\[\sigma_n^2:= \sum_{k \leq n} \int \psi_k^2 \circ T^k dm.\]

By Lemma \ref{close variance}, when $\alpha< \frac{1}{4}$,
\[\sigma_n^2 \succsim n^{\gamma} \iff \Sigma_n^2 \succsim n^{\gamma}. \]\newpage

Therefore, from now on, we assume
\[\sigma_n^2 \succsim n^{\gamma}.\]
\begin{lemma}\label{neighbound}
When $\alpha< \frac{1}{4}$, let $R_n:=\sum_{k \geq n} \frac{\psi_k\circ T^k}{\sigma_k^{2}}$, then
\[\frac{\sigma_{n+1}^2}{\sigma_{n}^2} \to 1, \frac{\int{R_n^2}dm}{\sigma_n^{-2}} \to 1.\] Furthermore, $ (R_n)_{n\ge 1}$ is reverse martingale w.r.t.
$ (T^{-n}\mathcal{B})_{n \ge 1}.$
\end{lemma}

\begin{proof}

When $\alpha < \frac{1}{4}$, by (\ref{bound3}), we have,
\[\frac{\sigma_{n+1}^2}{\sigma_{n}^2}= \frac{\sigma_{n}^2+ \int \psi_{n+1}^2 \circ T^{n+1} dm}{\sigma_{n}^2}=\frac{\sigma_{n}^2+O(1)}{\sigma_{n}^2} \to 1.\]

Since reverse martingale differences $ (\psi_k \circ T^k)_{k \ge 1}$ is orthogonal series, then
\[\int{ R_n^2 }dm= \sum_{k \geq n} \int {\frac{\psi_k^2\circ T^k}{\sigma_k^{4}}} dm =\sum_{k \geq n} \frac{\sigma_k^2-\sigma_{k-1}^2}{\sigma_k^{4}} \le \int_{\sigma_{n-1}^2}^{\infty}{\frac{1}{x^{2}}dx}=\sigma_{n-1}^{-2},\]
\[\mathbb{E}[R_n|T^{-(n+1)}\mathcal{B}]=R_{n+1}+\mathbb{E}[\frac{\psi_n\circ T^n}{\sigma_n^{2}}|T^{-(n+1)}\mathcal{B}]=R_{n+1},\]
\[\int{ R_n^2 }dm=\sum_{k \geq n} \frac{\sigma_k^2-\sigma_{k-1}^2}{\sigma_{k-1}^{4}} \cdot \frac{\sigma_{k-1}^{4}}{\sigma_{k}^{4}}\]
\[=\sum_{k \geq n} \frac{\sigma_k^2-\sigma_{k-1}^2}{\sigma_{k-1}^{4}}(1-o(1))^{2}\ge \int_{\sigma_{n-1}^2}^{\infty}{\frac{1}{x^{2}}dx}\cdot(1-o(1))^{2}.\]
Hence $\frac{\int{R_n^2}dm}{\sigma_n^{-2}}=\frac{\int{R_n^2}dm}{\sigma_{n-1}^{-2}} \cdot \frac{\sigma_{n-1}^{-2}}{\sigma_{n}^{-2}} \to 1, (R_n)_{n\ge 1} \text{ is reverse martingale}.$
\end{proof}

\subsubsection*{Skorokhod embedding for $R_n$}

\begin{lemma}[see \cite{SH}, Theorem 2]\label{embed}

  There are a constant $C>1$, optional times $ \tau_n \searrow 0$ and a
  Brownian motion $ (B_t)_{t \ge 0}$ on an extended probability space of
  $ ( [0,1], \mathcal{B}, dm)$ such that:
\begin{align}
&R_n=B_{\tau_n},\\
&\mathbb{E}[\tau_n-\tau_{n+1}|\mathcal{G}_{n+1}]=\mathbb{E}[\frac{\psi_n^2 \circ T^{n}}{\sigma_n^{4}}|T^{-(n+1)}\mathcal{B}], \text{ where } \mathcal{G}_n=\sigma(\tau_i, T^{-i}\mathcal{B}, i\geq n),\\
&\frac{1}{C} \cdot \mathbb{E}[(\tau_n-\tau_{n+1})^2|\mathcal{G}_{n+1}] \le \mathbb{E}[\frac{\psi_n^4\circ T^{n}}{\sigma_n^{8}}|T^{-(n+1)}\mathcal{B}] \le C \cdot \mathbb{E}[(\tau_n-\tau_{n+1})^2|\mathcal{G}_{n+1}].
\end{align}
\end{lemma}

\subsubsection*{Approximation for reverse martingale differences $(\psi_n
  \circ T^n)_{n \ge 1}$}\

\begin{lemma}\label{criteria}
  When $\alpha< \frac{1}{4}$, let
  $\delta_n^2:=\int{R^2_n}dm. \text{ If there is } \epsilon_0>0, \text{ such that
  }$
  \[\tau_n-\delta_n^2=o(\delta_n^{2+\epsilon_0}) \text{ a.s.}\]

  Then there is small $\epsilon'>0$ s.t.
  \[  |\sum_{i \leq n} \psi_i
    \circ T^i -\sum_{i \leq
      n}(B_{\delta_i^2}-B_{\delta_{i+1}^2})\cdot\sigma_i^2|=o(\sigma_n^{1-\epsilon'})
    \text{ a.s.}\]
\end{lemma}

\begin{proof}
By Lemma \ref{neighbound}, $\delta_n^2 \approx \sigma_n^{-2}$.
By Lemma \ref{embed} and definition of $R_n$, we have for any $i \ge 1$:
\[B_{\tau_i}=R_i=\sum_{k \geq i} \frac{\psi_k\circ T^k}{\sigma_k^{2}},\]
\[B_{\tau_{i}}-B_{\tau_{i+1}}=\frac{\psi_i\circ T^i}{\sigma_i^{2}}.\]

That is, for any $i \ge 1$:
\[\psi_i\circ T^i=(B_{\tau_{i}}-B_{\tau_{i+1}}) \cdot \sigma_i^{2}.\]

For $ m<n $, write $ \sum_{i \leq n} \psi_i \circ T^i$ as:
\[\sum_{i \leq m-1} \psi_i \circ T^i+ \sum_{m\leq i \leq n} \psi_i \circ T^i =\sum_{i \leq m-1} \psi_i \circ T^i+ \sum_{m\leq i \leq n}(B_{\tau_i}-B_{\tau_{i+1}})\cdot \sigma_i^{2}\]
\[= \sum_{i \leq m-1} \psi_i \circ T^i+ B_{\delta^2_m} \cdot \sigma_m^{2}-B_{\delta^2_{n+1}} \cdot \sigma_n^{2}+\sum_{m+1\leq i \leq n} B_{\delta_i^2}\cdot(\sigma_i^{2}-\sigma_{i-1}^{2})+ e_{m,n},\]
where the error term:
\[ e_{m,n}=\sum_{m+1\leq i \leq n}(B_{\tau_i}-B_{\delta_i^2})\cdot(\sigma_i^{2}-\sigma_{i-1}^{2})+(B_{\tau_m}-B_{\delta_m^2})\cdot\sigma_m^{2}-(B_{\tau_{n+1}}-B_{\delta_{n+1}^2})\cdot\sigma_n^{2}.\]

By H\"older continuity of Brownian motion near $t=0$, for any $ c< \frac{1}{2}$, fixed $ m \gg 1$:
\[ |e_{m,n}|\leq \sum_{m+1\leq i \leq n}|\tau_i-\delta_i^2|^c \cdot (\sigma_i^{2}-\sigma_{i-1}^{2})+|\tau_m-\delta_m^2|^c \cdot \sigma_m^{2}+|\tau_{n+1}-\delta_{n+1}^2|^c \cdot \sigma_n^{2}\]
\[\leq \sum_{m+1\leq i \leq n} o(\delta_i^{(2+\epsilon_0)c})\cdot (\sigma_i^{2}-\sigma_{i-1}^{2}) + o(\delta_m^{(2+\epsilon_0)c} ) \cdot \sigma_m^{2}+o(\delta_{n+1}^{(2+\epsilon_0)c})\cdot \sigma_n^{2} \precsim o( \sigma_n^{2-(2+\epsilon_0)c}).\]

We can choose $ c< \frac{1}{2}$, s.t. $2-(2+\epsilon_0)c<1 $, then there is small $\epsilon'>0$ s.t. $2-(2+\epsilon_0)c<1-\epsilon'$ and $ |e_{m,n}|=o(\sigma_n^{1-\epsilon'})  \text{ a.s.}$.

Therefore, $(\psi_n \circ T^n)_{n \ge 1}$ satisfies:
\[ |\sum_{i \leq n} \psi_i \circ T^i -\sum_{i \leq n}(B_{\delta_i^2}-B_{\delta_{i+1}^2})\cdot\sigma_i^2|=o(\sigma_n^{1-\epsilon'}) \text{ a.s.}\]
\end{proof}

\subsubsection*{ASIP for $(\phi_n\circ T^n)_{n\ge1}$}

\begin{lemma} \label{criteria1}
  $(\phi_n\circ T^n)_{n\ge1}$ satisfies ASIP if
  \[4 \alpha< \gamma,  \text{ and there is }
  \epsilon_0>0, \text{ s.t. } \tau_n-\delta_n^2=o(\delta_n^{2+\epsilon_0})
  \text{ a.s.}\]

\end{lemma}

\begin{proof}
From Lemma \ref{criteria} and (\ref{martingale}):
\[\sum_{i\le n} \phi_i \circ T^i=\sum_{i \leq n}(B_{\delta_i^2}-B_{\delta_{i+1}^2})\cdot\sigma_i^2+o(\sigma_n^{1-\epsilon'})+H_{n+1} \circ T^{n+1} \text{ a.s.}\]
By (\ref{bound2}), take $ r< \frac{1}{2\alpha}$, $r> \frac{2}{\gamma}$, there is $ \epsilon'>0 $ s.t. $\frac{ \gamma \cdot r}{2}\cdot(1-\epsilon')>1$ and
\[\int |\frac{ H_{n+1} \circ T^{n+1}}{\sigma_n^{(1-\epsilon')}} |^rdm \precsim \frac{1}{n^{\frac{\gamma \cdot r}{2}\cdot (1-\epsilon')}}.\]

That is, when $4\alpha< \gamma$, there is $r\in (\frac{2}{\gamma}, \frac{1}{2\alpha})$,  $ \epsilon'>0 $ s.t. $\frac{ \gamma \cdot r}{2}\cdot(1-\epsilon')>1$ and
\[\int |\frac{ H_{n+1} \circ T^{n+1}}{\sigma_n^{(1-\epsilon')}} |^r dm \precsim \frac{1}{n^{\frac{\gamma \cdot r}{2}\cdot (1-\epsilon')}}.\]

By the Borel-Cantelli Lemma:
\begin{equation}\label{error rate}
H_{n+1} \circ T^{n+1}=o(\sigma_n^{1-\epsilon'}) \text{ a.s. }
\end{equation}
If we define $G_i:=(B_{\delta_i^2}-B_{\delta_{i+1}^2})\cdot\sigma_i^2$, then by (\ref{error rate}) and Lemma \ref{close variance},
\[ |\sum_{i \leq n} \phi_i \circ T^i -\sum_{i \leq n}G_i|=o(\sigma_n^{1-\epsilon'})=o(\Sigma_n^{1-\epsilon'}) \text{ a.s.}\]

Besides, by Lemma \ref{close variance} again,
\[\sum_{i\le n} \mathbb{E} G_{i}^2=\sum_{i\le n} \mathbb{E}{[(B_{\delta_i^2}-B_{\delta_{i+1}^2})\cdot\sigma_i^2]^2} =\sum_{i\le n}(\delta_i^2-\delta_{i+1}^2)\cdot\sigma_i^4= \sum_{i\le n} \frac{\mathbb{E}\psi_i^2\circ T^i}{\sigma_i^{4}}\cdot \sigma_i^4\]
\[=\sigma_n^2=\Sigma_n^2+O(1)=\Sigma_n^2+O(\Sigma_n^{2(1-\epsilon')}).\]

Therefore, $(\phi_n\circ T^n)_{n\ge1}$ satisfies ASIP.
\end{proof}

\subsection*{Step 4: Estimates for ASIP}

From Lemma \ref{criteria1}, we only need to find the conditions for
$\gamma<1 $ and $\alpha<\frac{1}{4}$ in order that there is $\epsilon_0>0$ such that
\[ \tau_n-\delta_n^2=o(\delta_n^{2+\epsilon_0}) \text{ a.s.}\]

Decompose $ \tau_n-\delta_n^2 $ as three terms: $R^{'}_n+R^{''}_n+S_n$:
\[ R^{'}_n=\sum_{i \geq n} ( \tau_i - \tau_{i+1}-\mathbb{E}[\frac{\psi_i^2 \circ T^{i}}{\sigma_i^{4}}   |T^{-(i+1)} \mathcal{B}] ),\]
\[ R^{''}_n=\sum_{i \geq n} (\mathbb{E} [\frac{\psi_i^2 \circ T^{i}}{\sigma_i^{4}}|T^{-(i+1)} \mathcal{B}]-\frac{\psi_i^2 \circ T^i}{\sigma_i^{4}} ), \]
\[ S_n=\sum_{i \geq n} (\frac{\psi_i^2 \circ T^i}{\sigma_i^{4}}-\mathbb{E}\frac{\psi_i^2 \circ T^i}{\sigma_i^{4}}). \]

\subsubsection*{Estimate $R_n^{'}$ and $R_n^{''}$}

First note that $ R_n^{'}$, $ R_n^{''}$ are reverse martingales
with respect to filtrations $(\mathcal{G}_n)_{n \ge 1}$ and $ (T^{-n} \mathcal{B})_{n \ge 1}$ respectively:
\begin{lemma}
\[ \alpha < \frac{1}{8},\gamma > \frac{2}{3} \implies
  R_n^{'}=o(\delta_n^{2+\epsilon_0}),
  R_n^{''}=o(\delta_n^{2+\epsilon_0}).\]

\end{lemma}

\begin{proof}
  By (\ref{bound3}), $K_n := \sum_{i \leq n}\int{|\psi_i^4 \circ T^{i}|} dm \precsim
  n \precsim
  \sigma_n^{\frac{2}{\gamma}}$. By the martingale maximal inequality:
\[ \mathbb{E} |\frac{|\sup_{i \geq n}R'_i|}{\delta_n^{2+\epsilon_0}}|^2 \precsim \frac{1}{\delta_n^{2(2+\epsilon_0)}} \mathbb{E} {|R'_n|^2} \precsim \frac{1}{\delta_n^{2(2+\epsilon_0)}} \sum_{i \geq n}\int{\frac{\psi_i^4 \circ T^{i}}{\sigma_i^{8}}} dm  \]
\[ =\frac{1}{\delta_n^{2(2+\epsilon_0)}} \cdot(-\frac{K_{n-1}}{\sigma_n^{8}} + \sum_{i \geq n} K_i \cdot (\frac{1}{\sigma_i^{8}}-\frac{1}{\sigma_{i+1}^{8}}) ) \precsim \frac{1}{\delta_n^{2(2+\epsilon_0)}}\cdot (\frac{K_{n-1}}{\sigma_n^{8}} + \sum_{i \geq n} K_i \cdot \frac{(\sigma_{i+1}^{8}-\sigma_{i}^{8})}{\sigma_i^{16}}) \]
\[\precsim \frac{1}{\delta_n^{2(2+\epsilon_0)}}\cdot(\frac{1}{\sigma_n^{8-\frac{2}{\gamma}}} +\sum_{i \geq n}  (\frac{\sigma_{i+1}^{8}-\sigma_{i}^{8}}{\sigma_i^{16-\frac{2}{\gamma}}}) )\precsim \frac{1}{\delta_n^{2(2+\epsilon_0)}}\cdot(\frac{1}{\sigma_n^{8-\frac{2}{\gamma}}} + \int_{\sigma_n^{8}}^\infty \frac{1}{x^{\frac{16-\frac{2}{\gamma}}{8}}} dx).\]

When $ \gamma > \frac{1}{4}$, the last integral converges, hence:
\[ \mathbb{E} |\frac{|\sup_{i \geq n}R'_i|}{\delta_n^{2+\epsilon_0}}|^2 \precsim \frac{1}{ \sigma_n^{4-2\epsilon_0 -\frac{2}{\gamma}}} \precsim \frac{1}{n^{2\gamma-1-\epsilon_0\gamma}}.\]

Choose $ \omega>0, \text{ s.t. }\omega \cdot (2\gamma-1-\epsilon_0\gamma) > 1 $, by Borel-Cantelli Lemma:
\[ \sup_{i \ge \lfloor N^\omega \rfloor}R'_{i}=o(\delta_{\lfloor N^\omega \rfloor}^{2+\epsilon_0})\text{ a.s.} \]

For any $ n, \text{ there is } N$ s.t.
$ \lfloor N^\omega \rfloor \leq n \leq \lfloor (N+1)^\omega \rfloor$:
\[\frac{|R'_n|}{\delta_n^{2+\epsilon_0}} \leq  \frac{\sup_{i \geq \lfloor N^\omega \rfloor }|R'_i|}{\delta_{\lfloor N^\omega \rfloor}^{2+\epsilon_0}} \cdot \frac{\delta_{\lfloor N^\omega \rfloor}^{2+\epsilon_0}}{\delta_{n}^{2+\epsilon_0}}\le o(1)\cdot \frac{\delta_{\lfloor N^\omega \rfloor}^{2+\epsilon_0}}{\delta_{n}^{2+\epsilon_0}} \text{ a.s.} \]

Since $\alpha< \frac{1}{8}$, using (\ref{bound3}) and Lemma \ref{neighbound}, we have
\[ \frac{\delta_{\lfloor N^\omega \rfloor}^{2}}{\delta_{n}^{2}} \precsim \frac{\sigma_n^{2}}{\sigma_{\lfloor N^\omega \rfloor}^{2}}=\frac{\sigma_{\lfloor N^\omega \rfloor}^{2}+\sigma_{n}^{2}-\sigma_{\lfloor N^\omega \rfloor}^{2}}{\sigma_{\lfloor N^\omega \rfloor}^{2}} \precsim 1+ \frac{n-\lfloor N^\omega \rfloor}{\sigma_{\lfloor N^\omega \rfloor}^{2}}\precsim 1+ \frac{N^{\omega-1}}{N^{\omega \cdot \gamma}}.\]

Hence when $ \gamma \geq 1- \frac{1}{\omega}, \alpha < \frac{1}{8}$,  $\frac{\delta_{\lfloor N^\omega \rfloor}^{2+\epsilon_0}}{\delta_{n}^{2+\epsilon_0}}=O(1), |R'_n|=o(\delta_n^{2+\epsilon_0}) \text{ a.s.} $

If $\gamma> \frac{2}{3}$, we can find $\omega$ and small $\epsilon_0$ such that $ \omega \cdot (2\gamma-1-\epsilon_0\gamma) > 1,  \gamma \geq 1- \frac{1}{\omega}$ are all satisfied. Then
\[ R'_n=o(\delta_n^{2+\epsilon_0}) \text{ a.s.} \text{ when } \alpha < \frac{1}{8},\gamma > \frac{2}{3}.  \]

The estimate of $R^{''}_n$ is similar.
\end{proof}

\subsubsection*{Estimate $S_n$}
Denote
\[ S_n':=\sum_{i\leq n}(\psi_i^2 \circ T^i-\int \psi_i^2\circ T^idm). \]

\begin{lemma} \label{crucial estimate}
 $\text{ When } \alpha< \frac{1}{4}, \text{ if there is  } \epsilon'>0$,
 \[S_n'=o(\sigma_n^{2(1-\epsilon')})\text{ a.s.}\]
 then there is $\epsilon_0>0$ s.t.
 \[S_n=o(\delta_n^{2+\epsilon_0}) \text{ a.s.} \]
\end{lemma}

\begin{proof}
By Lemma \ref{neighbound}, $\sigma_n^2 \approx \sigma_{n+1}^2$. Take any $\epsilon_0 < 2 \epsilon'$, then we have
  \[S_n=\sum_{i \geq n} (\frac{\psi_i^2 \circ T^i}{\sigma_i^{4}}-\mathbb{E}\frac{\psi_i^2 \circ T^i}{\sigma_i^{4}})=\sum_{i \ge n} \frac{S_i'-S_{i-1}'}{\sigma_i^4}=-\frac{S_{n-1}'}{\sigma_n^4}+\sum_{i\ge n}S_i'\cdot (\frac{1}{\sigma_i^4}-\frac{1}{\sigma_{i+1}^4})\]
  \[\precsim \frac{1}{\sigma_n^{2+2\epsilon'}}+\sum_{i\ge n} \frac{\sigma_{i+1}^4-\sigma_i^4}{\sigma_i^{4\cdot \frac{6+2 \epsilon'}{4}}} \precsim \frac{1}{\sigma_n^{2+2\epsilon'}}+ \int^{\infty}_{\sigma_{n}^4} \frac{1}{x^{\frac{6+2\epsilon'}{4}}}dx\precsim \delta_n^{2+2\epsilon'}=o(\delta_n^{2+\epsilon_0}).\]
 \end{proof}

\subsubsection*{Decompose $S_n'$}

To estimate $ S'_n$: from the calculation on page 1140 in \cite{NTV}, it is the sum
of following five terms:
\begin{equation} \label{1}
\sum_{i\leq n}(\phi_i^2 \circ T^i-\int \phi_i^2\circ T^idm),
\end{equation}
\begin{equation} \label{2}
\int H_{n+1}^2 \circ T^{n+1} dm,
\end{equation}
\begin{equation} \label{3}
-H_{n+1}^2\circ T^{n+1},
\end{equation}
\begin{equation} \label{4}
\sum_{i \leq n }-2 \cdot (\psi_i \circ T^i \cdot H_{i+1}\circ T^{i+1}),
\end{equation}
\begin{equation} \label{5}
 \sum_{i \leq n} 2 \cdot (\phi_i\circ T^i \cdot H_i\circ T^i - \int \phi_i\circ T^i \cdot H_i\circ T^i dm).
\end{equation}

\subsubsection*{Estimate (\ref{1}):}

By Sprindzuk's Theorem in \cite{S},
\[(\ref{1}) \precsim n^{\frac{1}{2}}=o(n^{\gamma(1-\epsilon')})\le o(\sigma_n^{2(1-\epsilon')}) \text{ a.s. if } \gamma> \frac{1}{2}, \epsilon' \text{ is small.} \]

\subsubsection*{Estimate (\ref{2}):}

By (\ref{bound2}),
\[ (\ref{2})=O(1) \le o(\sigma_n^{2(1-\epsilon')}) \text{ a.s. if } \alpha < \frac{1}{4}, \epsilon' \text{ is small.}\]

\subsubsection*{Estimate (\ref{3}):}

By (\ref{error rate}),
\[ (\ref{3})=o(\sigma_n^{2(1-\epsilon')}) \text{ a.s. if } 4\alpha< \gamma, \epsilon' \text{ is small.}\]

\subsubsection*{Estimate (\ref{4}):}

First note that:
\[(\psi_i \circ T^i \cdot H_{i+1}\circ T^{i+1})_{i \ge 1} \text{ is reverse martingale difference w.r.t. }  (T^{-i}\mathcal{B})_{i \ge 1}.\]

Using (\ref{bound2}), (\ref{bound3}) when $\alpha< \frac{1}{8}$ and H\"{o}lder inequality, we have
\[ \int |\frac{\sum_{i \leq n }(\psi_i \circ T^i \cdot H_{i+1}\circ T^{i+1})}{ \sigma_n^{2(1-\epsilon')}}|^2 dm= \frac{\sum_{i \leq n }\int \psi^2_i \circ T^i \cdot H^2_{i+1}\circ T^{i+1}dm}{ \sigma_n^{4(1-\epsilon')}}  \]
\[\le \sum_{i\le n} \frac{||\psi^2_i \circ T^i||_{L^2} \cdot ||H^2_{i+1}\circ T^{i+1}||_{L^2}}{\sigma_n^{4(1-\epsilon')}}\precsim \frac{n}{n^{2\gamma(1-\epsilon')}}.\]

Choose $\omega >0$ s.t. $\omega \cdot (2 \gamma (1-\epsilon')-1)>1$, by Borel-Cantelli Lemma:
\[\lim_{N \to \infty} \frac{\sum_{i \leq \lfloor N^\omega \rfloor }(\psi_i \circ T^i \cdot H_{i+1}\circ T^{i+1})}{ \sigma_{\lfloor N^\omega \rfloor}^{2(1-\epsilon')}}=0 \text{ a.s.}\]

For any $n \in \mathbb{N}$, there is $N \in \mathbb{N}$ s.t. $\lfloor N^\omega \rfloor \le n \le \lfloor {(N+1)}^\omega \rfloor$. Then by Martingale inequality, (\ref{bound2}) and (\ref{bound3}) when $\alpha< \frac{1}{8}$ again, we have
\[\mathbb{E} \frac{\max_{\lfloor N^\omega \rfloor \le j\le \lfloor {(N+1)}^\omega \rfloor} |\sum_{j \le i \leq \lfloor (N+1)^\omega \rfloor}\psi_i \circ T^i \cdot H_{i+1}\circ T^{i+1}|^2}{\sigma_{\lfloor N^\omega \rfloor}^{4(1-\epsilon')}}\]
\[\precsim \frac{\mathbb{E}(|\sum_{\lfloor N^\omega \rfloor \le j \le i \leq \lfloor (N+1)^\omega \rfloor}\psi_i \circ T^i \cdot H_{i+1}\circ T^{i+1}|^2)}{N^{2 \gamma \omega(1-\epsilon')}} \]
\[\precsim \frac{\lfloor (N+1)^\omega \rfloor-\lfloor N^\omega \rfloor}{N^{2 \gamma \omega(1-\epsilon')}}\precsim\frac{N^{\omega-1}}{N^{2 \gamma \omega(1-\epsilon')}}.\]

Therefore if $2\gamma >1$, then we can choose small $\epsilon'$ s.t. $2\gamma \omega (1-\epsilon')+1-\omega>1$. By the Borel-Cantelli Lemma, we have
\[\max_{\lfloor N^\omega \rfloor \le j\le \lfloor (N+1)^\omega \rfloor} |\sum_{j \le i \leq \lfloor (N+1)^\omega \rfloor}\psi_i \circ T^i \cdot H_{i+1}\circ T^{i+1}|=o(\sigma_{\lfloor N^\omega \rfloor}^{2(1-\epsilon')}).\]

Besides, using (\ref{bound2}) and (\ref{bound3}) when $\alpha< \frac{1}{8}$ again, we have
\[\frac{\sigma_{\lfloor (N+1)^\omega \rfloor}^{2}}{\sigma_{ \lfloor N^\omega \rfloor }^{2}}=\frac{\sigma_{\lfloor N^\omega \rfloor }^{2}+\sigma_{\lfloor (N+1)^\omega \rfloor}^{2}-\sigma_{\lfloor N^\omega \rfloor }^{2}}{\sigma_{ \lfloor N^\omega \rfloor }^{2}} \precsim 1+ \frac{\lfloor (N+1)^\omega \rfloor- \lfloor N^\omega \rfloor }{\sigma_{ \lfloor N^\omega \rfloor}^{2}}\precsim 1+ \frac{N^{\omega-1}}{N^{\omega \cdot \gamma}}.\]

Hence when $ \gamma \geq 1- \frac{1}{\omega}$,  $\frac{\sigma_{\lfloor (N+1)^\omega \rfloor}^{2}}{\sigma_{ \lfloor N^\omega \rfloor }^{2}}=O(1)$. Therefore,
\[|\sum_{i \leq n}\psi_i \circ T^i \cdot H_{i+1}\circ T^{i+1}|\le |\sum_{i \leq \lfloor (N+1)^\omega \rfloor}\psi_i \circ T^i \cdot H_{i+1}\circ T^{i+1}|\]
\[+\max_{\lfloor N^\omega \rfloor\le j\le \lfloor (N+1)^\omega \rfloor} |\sum_{j \le i \leq \lfloor (N+1)^\omega \rfloor}\psi_i \circ T^i \cdot H_{i+1}\circ T^{i+1}|\precsim o(\sigma_{\lfloor (N+1)^\omega \rfloor}^{2(1-\epsilon')})+ o(\sigma_{\lfloor N^\omega \rfloor}^{2(1-\epsilon')}) \]
\[\precsim o(\sigma_{\lfloor N^\omega \rfloor}^{2(1-\epsilon')})\le o(\sigma_n^{2(1-\epsilon')}).\]

Therefore, when $\gamma> \frac{2}{3}, \alpha< \frac{1}{8}$, we can find small $\epsilon'$ such that $\gamma \geq 1- \frac{1}{\omega}, 2\gamma(1-\epsilon')-1> \frac{1}{\omega}$ hold and $(\ref{4})= o(\sigma_n^{2(1-\epsilon')}) $.

\subsubsection*{Estimate (\ref{5}):}

Let $ U_n:=(\ref{5})$. From the proof of Lemma
3.4 in \cite{NTV}: for $m< n $,
\begin{equation}\label{6}
\int|U_n-U_m|^2 dm  \precsim ||\phi||^2_{W^{1,\infty}}\cdot(n-m+\sum_{m \le j \le n}j^{\frac{2\alpha}{1-2\alpha}}).
\end{equation}

Although $ \phi$ is not $ C^1 $ in this paper, (\ref{6}) still holds for $ \phi \in \Lip[0,1]$ by the same argument in Corollary \ref{decay1}. Then, when $\alpha < \frac{1}{4}$:
\[  \int|\frac{U_n}{\sigma_n^{2(1-\epsilon')}} |^2 dm  \precsim \frac{1}{\sigma_n^{4(1-\epsilon')}} \cdot (n+ \sum_{ j \le n}j^{\frac{2\alpha}{1-2\alpha}} ) \precsim\frac{1}{n^{2\gamma(1-\epsilon')-\frac{1}{1-2\alpha}}}.\]
Choose $ \omega>0$ and small $\epsilon'$ s.t. $ \omega \cdot (2\gamma(1-\epsilon')-\frac{1}{1-2\alpha})>1$. By Borel-Cantelli Lemma:
\[U_{\lfloor N^\omega \rfloor}=o(\sigma_{\lfloor N^\omega \rfloor}^{2(1-\epsilon')}) \text{ a.s.}\]

For any $ n$, there is $N$ s.t. $\lfloor N^{\omega} \rfloor\le n \le \lfloor (N+1)^{\omega} \rfloor$, then:
\[|U_n| \le |U_{\lfloor N^{\omega} \rfloor}|+\sup_{\lfloor N^{\omega} \rfloor\le n \le \lfloor(N+1)^{\omega}\rfloor}|U_n-U_{\lfloor N^{\omega} \rfloor}|\]
\[= o(\sigma_{\lfloor N^\omega \rfloor}^{2(1-\epsilon')})+ \sup_{\lfloor N^{\omega} \rfloor\le n \le \lfloor (N+1)^{\omega}\rfloor}|U_n-U_{\lfloor N^{\omega}\rfloor}|.\]
Estimate $ \sup_{\lfloor N^{\omega}\rfloor\le n \le \lfloor(N+1)^{\omega}\rfloor}|U_n-U_{\lfloor N^{\omega}\rfloor}|
$:
\[ \int{| \frac{\sup_{\lfloor N^{\omega}\rfloor\le n \le \lfloor (N+1)^{\omega}\rfloor}|U_n-U_{\lfloor N^{\omega}\rfloor}|}{\sigma_{\lfloor N^\omega \rfloor}^{2(1-\epsilon')}} |^2 dm} \]
\[\le \frac{\sum_{\lfloor N^{\omega}\rfloor\le n \le \lfloor (N+1)^{\omega}\rfloor}\int{ |U_n-U_{\lfloor  N^{\omega}\rfloor}|^2 dm}}{\sigma_{\lfloor N^\omega \rfloor}^{4(1-\epsilon')}}  \]
\[\precsim\frac{1}{\sigma_{\lfloor N^\omega \rfloor}^{4(1-\epsilon')}} \sum_{\lfloor N^\omega \rfloor\le n \le \lfloor (N+1)^\omega \rfloor} (n-\lfloor N^{\omega}\rfloor +\sum_{\lfloor N^\omega \rfloor\leq j \leq n} j^{\frac{2\alpha}{1-2\alpha}}) \precsim \frac{N^{2(\omega-1)+\omega\frac{2\alpha}{1-2\alpha}}}{N^{2\gamma\omega(1-\epsilon')}}.\]

By Borel-Cantelli lemma, when $ \gamma > 1-\frac{1}{2\omega} + \frac{\alpha}{1-2\alpha}$, small $ \epsilon'>0$, then
\[2\gamma\omega(1-\epsilon')-2(\omega-1)-\omega\frac{2\alpha}{1-2\alpha}>1,\]
\[ \sup_{\lfloor N^{\omega} \rfloor \le n \le \lfloor (N+1)^{\omega}\rfloor}|U_n-U_{\lfloor N^{\omega} \rfloor }|= o(\sigma_{\lfloor N^\omega \rfloor}^{2(1-\epsilon')})\le o(\sigma_{n}^{2(1-\epsilon')}).\]

Therefore: \[(\ref{5})=U_n=o(\sigma_{n}^{2(1-\epsilon')}).\]

If $\gamma> \frac{1}{2}+ \frac{1+2\alpha}{4(1-2\alpha)}$, we can find $ \omega$ s.t. $ \omega \cdot (2\gamma-\frac{1}{1-2\alpha})>1$, $ \gamma > 1-\frac{1}{2\omega} + \frac{\alpha}{1-2\alpha} $ are all satisfied. Then by Lemma \ref{crucial estimate}, there is small $\epsilon_0>0$:
\[S_n=o(\delta_n^{2+\epsilon_0}) \text{ a.s.}  \text{ if } \gamma> \frac{1}{2}+ \frac{1+2\alpha}{4(1-2\alpha)}, \alpha< \frac{1}{8}.  \]

Therefore:
\[\tau_n-\delta_n^2=o(\delta_n^{2+\epsilon_0}) \text{ a.s. if } \gamma> \frac{1}{2}+ \frac{1+2\alpha}{4(1-2\alpha)}, \alpha< \frac{1}{8}.\]

By Lemma \ref{criteria1},
\[(\phi_n\circ T^n)_{n\ge1} \text{ satisfies ASIP, if } \gamma> \frac{1}{2}+ \frac{1+2\alpha}{4(1-2\alpha)}, \alpha< \frac{1}{8}.\]
\qed

\section{Applications of Theorem \ref{thm}}

\begin{theorem}[Nearby maps]

  Consider the non-stationary dynamical system\break
  $ ([0,1], \mathcal{B}, (T_k)_{k \ge 1}, dm)$ where $dm $ is Lebesgue
  measure, $T_0=T_{\beta_0}, T_k=T_{\beta_k}$ are Pomeau-Manneville maps,
  $ 0<\beta_0,\beta_k< \frac{1}{8} $ and an observation
  $\phi \in \Lip[0,1]$; assume $\phi $ is not co-boundary w.r.t. $ T_0$ in
  $ L^2([0,1],dm) $, i.e.
\[ \phi \neq c+ \psi \circ T_0-\psi \text{ for any measurable $\psi$ and constant c.}\]

Then there is $ \epsilon>0 $ such that
\[\text{for any } \beta_k \in (\beta_0-\epsilon, \beta_0+ \epsilon), \text{ } (\phi_i \circ T^i)_{i\ge 1}\text{ satisfies ASIP.} \]
\end{theorem}

\begin{proof}
  From \cite{NTV} Theorem 4.1, there is $ \eta>0$ s.t.
  $ ||\psi_n \circ T^n||_{L^2}> \eta$, for all $n \gg 1$, therefore
  $ \Sigma_n^2 \succsim n$ when $\beta_0, \beta_k < \frac{1}{4}$. By Theorem \ref{thm}, ASIP is satisfied.
\end{proof}

\begin{theorem}[Random compositions]
  Consider finitely many maps
  $([0,1], \mathcal{B},\break (T_k)_{0 \le k \le d}, dm)$, where $dm $ is
  Lebesgue measure, $T_k=T_{\beta_k}$, $0\le k \le d$, are
  Pomeau-Manneville maps, $ 0<\beta_k< \frac{1}{8} $, and an observation
  $\phi \in \Lip[0,1]$; assume $\phi $ is not co-boundary w.r.t. $ T_0$ in
  $ L^2([0,1],dm) $.

  Define a symbolic dynamical system
  $(\{ 0,1, \cdots, d\}^{\mathbb{N}}, \sigma, P^{\otimes \mathbb{N}})$
  where $ \sigma $ is the left shift and $ P $ is a probability on
  $\{ 0,1, \cdots, d\}$. Define random compositions by
  $ T^k_{\omega} := T_{(\sigma^{k-1} \omega )_0}\circ T_{(\sigma^{k-2}
    \omega )_0}\circ \cdots \circ T_{(\omega )_0}$ and random variance
  $ \Sigma_n^2(\omega)= \int (\sum_{i\le n}\phi \circ T^i_{\omega}-m(\phi
  \circ T^i_{\omega}) )^2 dm$ for
  $\omega\in \{ 0,1, \cdots, d\}^{\mathbb{N}}$. Then
  $ (\phi \circ T^n_{\omega}-m(\phi \circ T^n_{\omega}) )_{n\ge 1}\text{
    satisfies ASIP}$ for $ P^{\otimes \mathbb{N}}\text{-a.e. } \omega$.
\end{theorem}

\begin{proof}
  From \cite{NTV} Lemma 5.1, there is a $C>0$ and almost surely an
  $N_{\omega}\in \mathbb{N}$ s.t. $ \Sigma_n^2(\omega) \ge C n$, for all
  $n \ge N_{\omega}$ when $ \beta_k < \frac{1}{4}$. By Theorem \ref{thm}, ASIP is satisfied for each such
  $\omega$.
\end{proof}

\begin{remark}
Note that although \cite{NTV} proved that $ \Sigma_n^2$ (resp.
$\Sigma_n^2(\omega)$) has linear growth for $ \phi \in C^1[0,1]$, the
linear growth still holds for $ \phi \in \Lip[0,1]$ by the same argument in Corollary \ref{decay1}.
\end{remark}




\section*{Acknowledgments}
The author warmly thanks for the help from his advisor Prof. Andrew T\"or\"ok, and thanks University of Houston for good place to study dynamical system. The author also thanks referees for pointing out the typos and the error of the estimate of (\ref{4}).


\medskip
Received January 2019; revised   May 2019.
\medskip

\end{document}